\documentclass{article}

\usepackage{amsfonts,amsmath,amssymb}
\usepackage{graphicx}
\usepackage{authblk}
\newtheorem{satz}{Theorem}

\newtheorem{theorem}[satz]{Theorem}
\newtheorem{lemma}[satz]{Lemma}
\newtheorem{definition}[satz]{Definition}
\newtheorem{corollary}[satz]{Corollary}

\newcommand{\qed}{{} \hfill \mbox{$\Box$}}

\def\({\big (}
\def\){\big )}

\def\_phi{\varphi}

\title{Rank of matrices with entries from a multiplicative group}

\author
{Noga Alon
\thanks{  
Department of Mathematics, Princeton University, Princeton,
New Jersey, USA and Schools of Mathematics and Computer Science,
Tel Aviv University, Tel Aviv, Israel.
Email: nogaa@tau.ac.il}
\and
J\'{o}zsef Solymosi
\thanks{Department of Mathematics,
University of British Columbia, 1984 Mathematics Road,
Vancouver, BC, V6T 1Z2, Canada.
Email: solymosi@math.ubc.ca}}
\date{} % Activate to display a given date or no date
\begin{document}
\maketitle

\begin{abstract}
    We establish lower bounds on the rank of matrices in which all but 
the diagonal entries lie in a multiplicative group of small rank.
Applying these bounds we show that the distance
sets of finite pointsets in $\mathbb{R}^d$
generate high rank multiplicative groups and that multiplicative
groups of small rank cannot contain large sumsets.
\end{abstract}

\section{Introduction}
Bounding the rank of matrices satisfying  appropriate conditions
is an important topic in linear algebra. Such
bounds have various applications in divers areas of mathematics.
Several 
examples of applications in combinatorics and computer science
appear in the following papers of the first author
\cite{Alon1,Alon2}. In the present paper we combine some of the
techniques of these papers with additional number theoretic and
combinatorial tools in the derivation of lower bounds 
on the rank of matrices in which all but the diagonal entries lie
in a
multiplicative group of small rank. 

\medskip

Throughout the paper all matrices considered are $n\times
n$ complex matrices, unless otherwise specified. Our main result is
the following theorem

\begin{theorem}
\label{main}
For any positive integers $r$ and $D$ there is a threshold $n_0=n_0(r,D)$,
such that if $G$ is a multiplicative subgroup of $\mathbb{C}^*$ of rank
at most $r$ and $M=(m_{ij})$ is an $n\times n$ matrix, $n\geq n_0,$
where $m_{ij}\in G$ for every $i\neq j$ and $m_{jj}\not\in G$  ($1\leq
i,j\leq n$), then $rank(M)\geq D.$
\end{theorem}

We describe  two proofs of the theorem. 
The most important ingredient in both proofs is the Subspace Theorem for
linear equations with variables from a multiplicative group by Evertse,
Schlickewei and Schmidt \cite{ESS}. In addition, we use
an observation from the above mentioned paper \cite{Alon1} of the
first author and some additional 
simple combinatorial arguments and tools from linear algebra. 
As applications of Theorem \ref{main}
we prove that the distance sets of finite pointsets in $\mathbb{R}^d$
generate high rank multiplicative groups and that multiplicative
groups of small rank cannot contain large sumsets.

The rest of this short paper is organized as follows. In the next
section we describe the main ingredients of the proof of the main
result. Sections 3 and 4 contain two (similar) proofs of the result. 
The first is a bit simpler, the second provides a better
quantitative bound. Section 5 contains several applications and the final
section 6 contains some concluding remarks.

\section{The main tools}
In this section we describe the results which are the building blocks 
of the proof of Theorem \ref{main}.

\medskip

\begin{enumerate}
    \item The Subspace Theorem of Evertse, Schlickewei and Schmidt 
\cite{ESS}. We present the version with the best known 
bound due to Amoroso and Viada \cite{AV}.

\begin{theorem}
\label{Subspace}
Given an algebraically closed field $K$ and
a multiplicative subgroup $\Gamma$ of $K$ of finite rank $r$  in
it,
suppose $a_1, a_2,\ldots , a_m \in K^*.$ Then the
number of solutions of the equation

\begin{equation}
\label{lineq}
a_1z_1 + a_2z_2 + \ldots + a_mz_m = 1 
\end{equation}

with $z_i \in \Gamma$  
where no subsum on the left hand side vanishes is at most
$$A(m, r) \leq (8m)^{4m^4(m+mr+1)}\leq 2^{rm^5\log_c{m}},$$ for some 
absolute constant $c> 1.$
\end{theorem}

We apply the key feature of the theorem, that the bound $A(m,r)$
is a uniform bound, independent of the coefficients in
(\ref{lineq}). The Subspace Theorem is a powerful tool, it has several
important applications. For the interested reader we recommend the
excellent surveys by Bilu \cite{Bi} and by Bugeaud \cite{Bu}. For
some combinatorial applications see the survey by Schwartz and 
the second author \cite{SS}.

   \item The following well known bound for the multicolor
Ramsey numbers for complete graphs follows from the
neighborhood-chasing argument
in the classical paper of Erd\H{o}s and Szekeres \cite{ESz}.

\begin{theorem}\label{Ramsey}
For a positive integer $t$ let $R(t,\ell)$ denote the least integer such
that any $\ell$-coloring of the edges of the complete graph on $R(t,\ell)$
vertices contains a monochromatic complete subgraph of size $t.$ Then
$R(t,\ell)\leq \ell^{\ell t}.$
\end{theorem}

\iffalse
The proof goes as in \cite{ESz}. Choose an arbitrary vertex, $v,$ in an
$n$-element complete graph. At least $1/\ell$-fraction of the neighbours
of $v$ are connected by the same colored edges, these edges have color
$i,$ say. These neighbors are denoted by $N_1.$ Color $v$ by $i,$ and
repeat the process inside $N_1.$ Vertices receiving the same color form
a monochromatic complete subgraph. In a graph of size $N$ the number
of steps, $j,$ before we stop ($|N_j|= 1$) is at least $\log_\ell N$
and at least $(\log_\ell N)/\ell$ vertices colored by one of the colors.
\fi

    \item Rank of matrices with few distinct entries in 
the lower triangular part outside the diagonal

\begin{theorem}\label{matrix}

Let $A$ be an $n\times n$ matrix where every row has at most $s$ distinct
values under the diagonal, and the element in the diagonal is different
from the elements in the row under the diagonal. If the rank of $A$
is $\varrho$ then
\[n\leq \binom{\varrho+s}{\varrho}.  \]

\end{theorem}

To prove the theorem, let us recall a lemma from
\cite{Alon1} with its proof. In our second proof of Theorem
\ref{main} we describe a generalization of this lemma which
enables us to avoid the Ramsey argument.

\begin{lemma}
\label{Noga}
Let $B = (b_{i,j})$ be an $n$ by $n$ matrix of rank $d,$ and let $P(x)$
be an arbitrary polynomial of degree at most $k.$ Then the rank of the $n$
by $n$ matrix $(P(b_{i,j}))$ is at most $\binom{k+d}{k}.$ If $P(x)=x^k$
then the rank is at most $\binom{k+d-1}{k}.$
\end{lemma}

\begin{proof}
Let ${\bf v_1} = (v_{1,j})^n_{j=1},$ ${\bf v_2} = (v_{2,j})^n_{j=1},$
$\ldots ,$ ${\bf v_d} = (v_{d,j})^n_{j=1}$ be a basis of the row-space
of $B$. Then the vectors $(v^{k_1}_{1,j}\cdot v^{k_2}_{2,j} \ldots \cdot
v^{k_d}_{d,j})^n_{j=1},$ where $k_1, k_2, \ldots , k_d$ range over all
non-negative integers whose sum is at most $k,$ span the rows of the
matrix $(P(b_{i,j})).$ If we have $P(x)=x^k$ then we only have to use
the exponents whose sum is exactly $k.$
\qed
\end{proof}

\medskip
\begin{proof} (of Theorem \ref{matrix})
Note that while Lemma \ref{Noga} is stated for a single polynomial
$P(x),$ it is used independently in every row. So, the same result
holds if one applies different, degree $\leq k$, polynomials in every
row. To prove Theorem \ref{matrix} let us define a polynomial for every
row. For row $i$, if the distinct elements under the diagonal are
denoted by $\alpha_1,\ldots,\alpha_m$ then define $P_i(x)$ by
$P_i(x)=\prod_{i=1}^m (x-\alpha_i) .$  Every polynomial has degree at
most $s,$ and the matrix, after applying the polynomials row-wise,
is a matrix with non-zero diagonal entries and zeros under the
diagonal, so it has full rank. If $A$ had rank $\varrho$ then $n\leq
\binom{\varrho+s}{\varrho}.$
\qed
\end{proof}

\item
The rank of Hadamard products of matrices.

For two matrices $A=(a_{ij})$ and $B=(b_{ij})$ with the same
number of rows and columns, the {\em Hadamard product} (the
element-wise product) of $A$ and $B$, denoted
by $A\bullet B$, is the matrix $A\bullet B=(a_{ij} \cdot b_{ij})$.
The following property of this product, mentioned by Ballantine
in \cite{Ba} , follows from the fact that
$A\bullet B$ is a submatrix of the tensor product of $A$ and $B$,
whose rank is the product of the ranks of the two matrices.
\begin{lemma} 
\label{Had}
For any two matrices $A$ and $B$ of the same dimension,
$rank(A\bullet B)\leq rank(A)\cdot rank(B). $  
\end{lemma}
\end{enumerate}

\medskip{}
After collecting the main required ingredients 
we are ready to prove Theorem \ref{main}.

\section{First proof of Theorem \ref{main}}
We start with a rough outline of the proof. 
First we choose a subset of row vectors forming a basis, $B,$ of the
row space of $M.$ Adding any other row vector to the basis, there is
a nontrivial linear form, $\Lambda$ of the vectors in $B$ and the
new row giving the zero
vector. Checking the linear combinations coordinate-wise, we are hoping
to have many equations of the form like in (\ref{lineq}).  We partition
the elements of the matrix based on the subset of $B$ which gives a zero
sum in $\Lambda$ without zero subsums. 
Using the Subspace Theorem while focusing on a submatrix chosen by 
an appropriate application of Ramsey's Theorem we
bound the number of
distinct elements below the diagonal in each row of the submatrix.
As the last step we apply the rank bound from
Theorem \ref{matrix}. The detailed argument follows.

\medskip

Let $d$ denote the rank of $M$ and let $B=\{{\bf v_1} =
(v_{1,j})^n_{j=1},$ ${\bf v_2} = (v_{2,j})^n_{j=1},$ $\ldots ,$ ${\bf v_d}
= (v_{d,j})^n_{j=1}\}$ be a basis of the row-space of $M$. Without
loss of generality assume this basis consists of the first 
$d$ rows of $M$.
If ${\bf w}
= (w_{j})^n_{j=1},$ is any other row of $M$, outside the basis, 
then there is a
linear form, $\Lambda,$ with coefficients $c_0\neq 0,c_1,c_2,\ldots ,c_d$
such that

\begin{equation}
\label{first_lin}
c_0{\bf w}+c_1{\bf v_1}+\ldots+c_d{\bf v_d}={\bf 0}.
\end{equation}

In this vector equation let us consider the $n-d-1$ equations out of
the $n$ coordinate-wise equations, where none of the diagonal elements
appears. In the $i$-th coordinate of the vector equation
there is a nonempty
index set $I\subset [d]$, so that we have an equation of the form

\begin{equation}
\label{lineq2}
    c_0w_i+\sum_{\ell\in I\subset [d]}c_\ell v_{\ell i}=0
\end{equation}
without any subsum adding up to zero.  (Note that $w_i \neq 0$ 
as it belongs to the multiplicative subgroup $G$). 

\medskip
We label the matrix element $w_i$ with an element of the index set
$I.$ (We can choose, for example, the first element of $I.$) 
In this way any non-diagonal
element of the matrix outside of the coordinates of the basis
receives
a label, an element of $[d].$ Now we are looking for a large principal
submatrix \footnote{a submatrix sharing the diagonal with $M$} with the
same indices under the diagonal.

\medskip
The lower triangular submatrix of the labels (with zeros
in the diagonal), can be viewed as the edge coloring of a complete
graph on $n-d$ vertices with at most  $d$ colors. By Theorem
\ref{Ramsey} there is a principal submatrix
of $M,$ denoted $U$, of size at least
\[
\frac{\log{(n-d)}}{d\log{d}}.
\]

In $U,$ every element under the diagonal has the same label, $\ell\in
[d]$ . Before we can apply the Subspace Theorem, we need one
additional
step, as follows. Divide every element $w_i$ of $U$ that belongs to 
column number $j$ of the original matrix $M$
by $c_\ell v_{\ell j}$, where $c_{\ell}$ is the coefficient of
$v_{\ell}$ in the expression (\ref{lineq2}). 
(Note that $c_\ell v_{\ell j} \neq 0$ as no subsum in 
\ref{lineq2} is $0$). 
The modified submatrix obtained this way from $U$ is denoted by $U'.$ 
Its rank is at most $d$, as it is obtained from $U$ (whose rank is
at most $d$) by first
dividing every column by a constant, and then by dividing every 
row by a constant. Note also that after dividing the column number
$j$ by $v_{\ell j}$, all non-diagonal elements of the column belong
to the multiplicative subgroup $G$ while the diagonal element is
not in $G$. Therefore, even after dividing the row by $c_\ell$, the
diagonal element stays different from the non-diagonal ones in the
row.  Consider the entries under the diagonal in $U'.$ 
If ${\bf u}$ is a row vector of $U'$ then there are
coefficients $a_1,\ldots, a_d$ such that if
$u_i$ is a  coordinate under the diagonal,
then there is an index set $J\subset [d]\setminus\ell$ such that

\begin{equation}\label{lineq3}
    a_0u_i+\sum_{j\in J}a_j v_{ji}=1,
\end{equation}
and no subsum on the left side is zero. We partition the coordinate-wise
equations for the selected ${\bf u}$ into no more than $2^{d-1}$ classes
based on the subset $J\subset [d].$ Every non-diagonal 
element in row ${\bf
u}$  satisfies the equation (\ref{lineq3}) for some index set $J.$ 
By Theorem \ref{Subspace} we know that for $|J|\geq 1$
there are no more than  $A(|J|+1,r)$ nontrivial solutions for the linear
equation in (\ref{lineq3}) with $u_i, v_{ji}$ in the multiplicative
subgroup $G$, therefore row $\bf u$ contains no more than
$2^{d-1}A(d,r)$ distinct entries under the diagonal. (If
$J=\emptyset$ then there is only one solution to (\ref{lineq3}).) 
Applying
Theorem \ref{matrix} we conclude that since the rank of  $M$ 
is $d$ then
\[{\frac{\log{(n-d)}}{d \log{d}} }\leq 
\binom{d+2^{d-1}A(d,r)}{d}\leq 2^{rd^6\log_{c}{d}},
\]
with some absolute constant $c>0.$ Therefore, if $n$ is
sufficiently large then 
\[
\log\log{n}\leq rd^{7}.
\]

This completes the proof of Theorem \ref{main}. 
\qed

\medskip
The quantitative bound we get from the proof is quite weak.
It would be very interesting to get better bounds even for cases where
the rank of the multiplicative group is very small, for example when
all non-diagonal elements are powers of two. In the next section we 
improve the bound by avoiding the application of Ramsey's Theorem. The
proof is similar though slightly more complicated, and gives a better 
bound.

\section{An improved bound}

The reason we had to apply Ramsey's Theorem (Theorem \ref{Ramsey})
in the proof of Theorem \ref{main} is
that in a row ${\bf w},$ while all coordinate entries satisfy equation
(\ref{first_lin}), it might be that there is a zero subsum with $w_i$,
so the Subspace Theorem is not directly applicable. We have thus selected a
principal submatrix where the entries had the same label, enabling
us to apply the Subspace Theorem. In order to improve the
quantitative estimate we replace the application of
Ramsey's Theorem by a linear algebra  argument based on Lemma
\ref{Had}. This enables us to
record all the required information without the consideration of
small submatrices of $M$. We need the following extension of
Lemma \ref{Noga}.

\subsection{Rank under pointwise application of multivariate polynomials}

Let ${\bf z_i}=(z_{i1}, z_{i2}, \ldots ,z_{in})$, $1 \leq i \leq r$,  be 
$r$ vectors over a field $F$, and let $Q(x_1,x_2, \ldots ,x_r)$ be a
multivariate polynomial in $F[x_1,x_2, \ldots x_r]$. The vector
${\bf u}=Q(\bf z_1,\bf z_2,\ldots, \bf z_r)$ is the vector  
${\bf u}=(u_1,u_2, \ldots ,u_n)$ defined by
$u_j=Q(z_{1j},z_{2j}, \ldots ,z_{rj})$.  Thus $u$ is obtained by applying
the polynomial $Q$ to the vectors $\bf v_i$ coordinate-wise.
\begin{theorem}
\label{t11}
Let $A_1, A_2, \ldots ,A_r$ be $r$ matrices over a field $F$, where 
each $A_i$ has $n$ columns, and let $d_i$ be the rank of $A_i$.
Let $A$ be a matrix with $n$ columns in which every row ${\bf u}$ is 
$Q_{\bf u}(\bf z_1, \bf z_2, \ldots , \bf z_r)$ where $\bf z_i$ is some
row of the matrix  $A_i$ and $Q_{\bf u}$ is some polynomial in
$F[x_1,x_2, \ldots ,x_r]$. If the degree of each of the polynomials
$Q_{\bf u}$ in $x_i$ is at most $k_i$, then the rank of $A$ is at most
$$
\prod_{i=1}^r {{k_i+d_i} \choose {d_i}}.
$$
\end{theorem}

\begin{proof}
If  $r=1$ the result is proved in Lemma \ref{Noga} following the
argument in \cite{Alon1}).
Therefore, if
$B_i$ is  the matrix whose rows are all vectors of the form
$Q(\bf v_i)$ with $\bf v_i$ being a row of $A_i$ and $Q$ being
a polynomial  in $\{1,x,x^2, \ldots ,x^{k_i}\}$ then the rank of
$B_i$ is at most $ {{k_i+d_i} \choose {d_i}}$.
By Lemma \ref{Had} the linear space spanned by all 
Hadamard products of one vector 
from each $B_i$ has dimension at most
$$
\prod_{i=1}^r {{k_i+d_i} \choose {d_i}}.
$$
Every row of $A$ lies in this linear space, implying 
the desired result. \qed
\end{proof}

\medskip

%%%%%%%%%%%%%%%%%%%%%%%%%%%%
\subsection{A modified proof of Theorem \ref{main}}

Let $M=(m_{ij})$ be a matrix satisfying the assumptions of Theorem
\ref{main}, let $d$ denote its rank, and assume, as before, that 
the first $d$ rows of $M$ form a basis of its row-space.
Denote these rows by 
$$
\{{\bf v_1} = (v_{1,j})^n_{j=1}, {\bf v_2}
= (v_{2,j})^n_{j=1}, \ldots ,{\bf v_d} = (v_{d,j})^n_{j=1}\}.
$$ 
Define $d$ matrices $M_1,M_2, \ldots ,M_d$, each having
$n-d$ rows and $n-d$ columns, as follows. For each $1 \leq \ell \leq
d$ the matrix $M_{\ell}$ has its rows and columns indexed by the integers
$j$ satisfying $d< j \leq n$. 
The element $M_{\ell}(i,j)$ is
defined as $m_{ij}/v_{\ell j}$. Thus the column with index $j$ of 
$M_{\ell}$ is obtained from the corresponding column of the matrix 
obtained from $M$ by deleting its first $d$ rows, by
dividing all elements of this column by $v_{\ell j} =m_{\ell j}$. 
Note that each $m_{\ell j}$ is nonzero, as it belongs to the
multiplicative subgroup $G$. It is clear that the rank of each
matrix $M_{\ell}$ is at most $d$, which is the rank of $M$.
We next define $d$ matrices $A_1,A_2, \ldots ,A_d$ of the same
dimension as the matrices $M_{\ell}$, 
where each
row of $A_{\ell}$ is a multiple of the corresponding row of 
$M_{\ell}$, as follows. 
Let ${\bf w}$ be an arbitrary row of the original matrix $M$ which is
not among the first $d$ rows $\bf v_i$.  Then ${\bf w}$ satisfies an
equation of the form (\ref{first_lin}). If $c_{\ell} \neq 0$ then
the row of $A_{\ell}$ corresponding to $\bf w$ is obtained 
from the corresponding row of $M_{\ell}$ by dividing it  by 
$c_{\ell}$. Otherwise (that is, if $c_{\ell}=0$) let this row
equal the corresponding row of $M_{\ell}$ as it is. It is clear
that the rank of each of the matrices $A_{\ell}$ is at most $d$.

We next define for each row ${\bf w}$ and each index $\ell$ a set
$S_{{\bf w},\ell}$ of at most $2^{d-1} A(d,r)$ elements so that the
following holds.  
\begin{enumerate}
\item
For every row ${\bf w}$ the diagonal coordinate of it
in each of the matrices $A_{\ell}$ does not belong to
$S_{{\bf w},\ell}$.
\item
For every row ${\bf w}$ and every non-diagonal element  
of it there exists at least one index $\ell$ so that
it belongs to the set $S_{{\bf w},\ell}$.
\end{enumerate}

The sets $S_{\bf w, \ell}$ are defined 
using the Subspace  Theorem, by repeating the
arguments in the previous proof of the theorem.
Indeed, each non-diagonal coordinate
$w_i$ of ${\bf w}$ for  $i>d$ satisfies an equation of the form
(\ref{lineq2}) without any subsum adding up to zero. For each such
coordinate there is at least one index $\ell$ so that 
$c_{\ell} \neq 0$.  We partition the coordinates according to the
specific subset of indices $I$ (which contains $\ell$) 
and use the Subspace Theorem to conclude that
the total number of distinct values of coordinates in the row with 
this set of indices $I$ 
is at most $A(d,r)$. The union, over all $2^{d-1}$ 
such subsets, of all these sets of values, is the set
$S_{\bf w,\ell}$.  It is clear that it satisfies property (2) above.
In addition, each diagonal element of every matrix $M_{\ell}$ differs from 
all non-diagonal  elements in the same row, as the non-diagonal 
elements lie in the 
group $G$ whereas the diagonal ones do not. Therefore each
diagonal element of every matrix $A_{\ell}$  differs from all
non-diagonal elements of this matrix in the same row, implying that
property (1) holds as well.

Finally we define, for each row ${\bf w}$,  the following
polynomial
$$
Q_{\bf w}(x_1,x_2, \ldots ,x_d)=
\prod_{\ell=1}^d \prod_{s \in S_{\bf w, \ell}} (x_\ell-s).
$$
In the notation of Theorem \ref{t11}, 
$Q_{\bf w} (\bf z_1,\bf z_2, \ldots ,\bf z_d)$, where $\bf z_\ell$ 
is the row 
corresponding to ${\bf w}$ in the matrix $A_{\ell}$,
is a vector whose only nonzero coordinate is in the diagonal.
By Theorem \ref{t11} the rank of the $(n-d) \times (n-d)$ matrix
consisting of all these rows, which is $n-d$, is
at most $ {{d+2^{d-1}A(d,r)} \choose d}^d$.  Therefore
$$
n-d \leq {{d+2^{d-1}A(d,r)} \choose d}^d \leq 2^{rd^7 \log_c d}
$$ 
for some absolute constant  $c>1$. It follows that for sufficiently
large $n$, $\log n \leq rd^8$ implying that
$d \geq (\frac{\log n}{r})^{1/8}.$ 

This completes the proof of the theorem with the improved bound.
\qed

\section{Applications}

\subsection{Sumsets in multiplicative groups}

The Subspace Theorem has been used in Additive Combinatorics in problems
related to the Sum-Product problem, showing the ``incompatibility''
of multiplicative and additive structures. Such applications started
with the paper of Chang \cite{Ch} where she proved that sets with small
product set have large sumsets. For a finite set, $A\subset \mathbb{C},$
the sumset is defined as $\{A+A\}=\{a+b : a,b\in A\}.$ The difference set
and product set are defined in the same way, one considers the pairwise
differences and products. Roche-Newton and Zhelezov \cite{RZ} proved that
multiplicative subgroups  $\Gamma\subset \mathbb{C}^*$ with small rank
cannot contain large difference sets. There is a function $f(x)$ such that
if $rank(\Gamma)\leq r$ and $\{A-A\}\subset \Gamma$ then $|A|\leq f(r).$
In their proof they also applied the Subspace Theorem. Here we prove
the similar statement for $\{A+A\}.$ For this we need an extension of
Theorem \ref{main}. In the last step of the proof we only used that the
diagonal elements are different from the other elements in that
row (or only those under the diagonal, in the first proof). 
The only step where we changed a diagonal entry
(without changing the other elements of the row in the same way)
was when we multiplied every element of column $j$ by $v_{\ell j}^{-1}.$
Therefore in the theorem we can replace the condition that diagonal
elements are  not from $G$ by a weaker one.

\begin{definition}
An $n\times n$ matrix with elements $\{a_{ij}\}_{i,j=1}^n$ satisfies the
rectangle condition if for any $i<j\neq k$ indices $a_{jj}a_{ik}\neq
a_{j,k}a_{ij}.$
\end{definition}

\begin{theorem}
\label{main_2}
For any positive integers $r$ and $D$ there is a threshold $n_0=n_0(r,D)$,
such that if $G$ is a multiplicative subgroup of $\mathbb{C}^*$ of rank
at most $r$ and $M=(m_{ij})$ is an $n\times n$ matrix, $n\geq n_0,$
where $m_{ij}\in G$ for every $i\neq j$ and $M$ satisfies the
rectangle condition, then $rank(M)\geq D.$
\end{theorem}

\begin{corollary}
There is a function $f(x)$ such that if $rank(\Gamma)
\leq r$ and $\{A+A\}\subset \Gamma$ then $|A|\leq f(r).$
\end{corollary}

\begin{proof}
If the elements of $A$ are denoted by $\{a_1,\ldots,a_n\}$ then we define
a matrix $M$ by $m_{i,j}=a_i+a_j.$  The rank of M is at most two. All
we have to check is that $M$ satisfies the rectangle condition. The
equation $(x+x)(y+z)=(x+z)(y+x)$  has only solution when $y=x$ or $z=x,$
but these numbers are distinct.
\qed
\end{proof}

\subsection{Multiplicative groups generated by distance sets}

As another application of Theorem \ref{main} we prove that the
distance sets of finite pointsets in $\mathbb{R}^d$ generate high rank
multiplicative groups.

\begin{theorem}
\label{points1}
For any positive integers $r,d$  there is a bound $N=N(r,d)$, such that
if $G$ is a multiplicative subgroup of $\mathbb{R}^*$ of rank at most $r$
and there are $n$ points in $\mathbb{R}^d$ where the pairwise distances
are from $G$ for every pair of points then $n\leq N.$
\end{theorem}

\begin{proof}
Suppose that there are $T$ points in $\mathbb{R}^d$,
$\{p_1,\ldots,p_T\}$. Let us consider the $T\times T$ matrix, $\Delta$,
where the $\delta_{i,j} $ entry is the square of the the  distance between
$p_i$ and $p_j.$ The diagonal of $\Delta$ contains zeros only, and the
other entries are positive real numbers. The entries are images of
a quadratic polynomial with $2d$ variables, and the 
rank of the matrix is at
most $d+2,$ since it can be written as the linear combination of $d+2$
rank one matrices.
\[
\Delta=X^{(2)}-2\sum_{k=1}^d XY(k)+Y^{(2)}.
\]
Here every entry in the $i$-th row of $X^{(2)}$ is the sum of the
squares of coordinates of $p_i$ and every entry in 
the $j$-th column of $Y^{(2)}$
contains the sum of squares of the coordinates of $q_j.$ In the
$\{i,j\}$ position of $XY(k)$ we have the product of the $k$-th
coordinates of $p_i$ and $q_j.$ Since all diagonal
elements are $0$, Theorem \ref{main} implies that if the rank
of the multiplicative group generated by the non-diagonal elements of
the matrix is at most $r$ then the (matrix) rank of $\Delta$ is at
least as $(\log{T}/r)^{1/8},$ which is larger than $d+2$
for large enough $T.$
\qed
\end{proof}

\subsection{Integral distances}
By Theorem \ref{points1} if all distances determined by a 
set of more than $N(r,d)$ points in $\mathbb{R}^d$ are integers then
there are at least $r+1$ distinct primes that divide at least one of
these distances. Here, however, we can prove a stronger result, with
a much better bound. Before stating and proving it we include a
brief discussion of some of the background about sets determining
integer distances, which are sometimes called {\em integral}
pointsets.

In 1945 Anning and Erd\H{o}s proved in
\cite{AE} (see also in \cite{E}) that if in a set of points in the plane
all pairwise distances are integers 
then the pointset is finite, or all points are on a line.
They asked if there are arbitrarily large integral pointsets in the
plane with no three on a line and no four on a circle. The problem
is still widely open, the best construction is due to Kreisel and Kurz
\cite{KK}, who found seven points using computer search. Another related
question is the Erd\H{o}s-Ulam conjecture, that there are no everywhere
dense pointsets in the plane such that all pairwise distances are
rational. This is also open although there are works showing that the
existence of such sets would contradict the Bombieri-Lang conjecture
\cite{T,Sa,ABT} and the abc conjecture as well \cite{Pa}. In the plane
the diameter of large integral pointsets should be large \cite{Soly,Ku,Av},
but not much is known about the structure of such sets. For dimension
$d>2,$ Nozaki proved that an $n$-element integral pointset has diameter
at least $n^{1/d}$ \cite{N}. As an application of our results and
techniques here
(with a much simpler proof and an improved bound that holds in this
case)
we show that for any integral pointset of $n$ points in $\mathbb{R}^d$
and any prime $p$ smaller than $n^{1/(d+1)}$, the pointset must determine
a distance divisible by $p$.

\begin{theorem}
\label{points2}
For any positive integer $d$ and any prime $p$ there is a threshold, 
$T=T(d,p) \leq {{p+d} \choose {d+1}}+1$
such that any set of more than  $T$ points in $\mathbb{R}^d$
in which all pairwise distances are integers, determines a distance
divisible by $p$.
\end{theorem}

Note that the theorem is not an empty statement, there are arbitrarily
large sets with integer distances. Even in the plane one can find large
sets on a circle such that all pairwise distances are integers (see
\cite{SdZ} for some constructions).

\medskip

\begin{proof}(of Theorem \ref{points2}) 
If $p$ is the smallest prime that does not divide any of the distances
then apply the polynomial $x^{p-1}$ to every entry of the matrix of
the squares of distances. This keeps all diagonal elements $0$, 
and changes
every non-diagonal entry, $d_{ij},$ to $d_{ij}^{p-1} \equiv 1 (\mod p).$
The rank of this matrix over $\mathbb{Z}_p$ is at least $T-1,$
showing that $${\binom{d+1+p-1}{p-1}}=\binom{p+d}{d+1} \geq T-1.$$
Here we applied the slightly better bound from Lemma \ref{Noga}, using that
$x^{p-1}$ is a special polynomial.
\qed
\end{proof}

\section{Concluding remarks and open problems}

\begin{itemize}
\item
Both proofs given here for Theorem \ref{main} provide weak
quantitative bounds, and it will be interesting to improve them.
As mentioned in Section 3, even the very special case of
determining the minimum possible rank of an $n$ by $n$ matrix in
which all non-diagonal elements are powers of $2$, and all diagonal
elements are not, is intriguing. By (a very special case of) 
Theorem \ref{main} (with $r=1$) 
this  minimum tends to infinity with $n$, but the lower bound 
obtained is probably very far from being tight.
The following example shows that this
minimum is at most $O(n^{1/3})$. Put  $m=3^d+1$, let
$P=P_3^d$ be the space of 
all vectors of length $d$ over $F_3$ and let $z$ be an additional
point.
Let $F$ be the collection of all planes in $P$, that is, all
the $2$-dimensional affine subspaces, and let $F'$ be the collection 
of all sets $L \cup \{z\}$ where $L \in F$. Note that each member
of $F'$ is of cardinality $9+1=10$. The intersection of every pair of distinct
members of $F$ is either empty, or a point, or a one-dimensional line,
and therefore the cardinality of each such intersection lies in the set
$\{0,1,3\}$. It follows that the cardinality of the intersection of 
every pair of distinct members of $F'$ is in the set $\{1,2,4\}$,
that is, it is a power of $2$. The Gram matrix of the
characteristic vectors of the elements of $F'$ is an
$|F'|$ by $|F'|$ matrix in which
all diagonal elements are $10$ and every
non-diagonal element is  a power of $2$. The rank of this matrix 
is clearly at most $m=3^d+1$ and its size is
$|F'|=\frac{(3^d)(3^d-1)(3^d-3)}{3^2(3^2-1)(3^2-3)}=\Omega(m^3)$.
\item
The statement of Theorem \ref{points1}  holds if we only assume
that the squares of the pairwise distances between pairs of points
belong to the multiplicative group $G$. This clearly follows from
the proof. Similarly the statement of Theorem \ref{points2} holds 
if we merely assume that the squares of the distances are integers.
\item
Theorem \ref{points2} can be extended
to prime powers.  
That is, every prime power $q=p^k$
divides some square of a distance
determined by any set of more than ${{q+d+1} \choose {d+1}}$ points in
$\mathbb{R}^d $ in which all squares of distances between pairs 
are integers.
The proof follows that of Theorem \ref{points2}, the only
difference is that instead of the polynomial $x^{p-1}$ we use here
the polynomial ${{x-1} \choose {q-1}}$.
By the theorem of Lucas \cite{Lu} the value of this polynomial is  not
$0$ modulo $p$ if and only if $x$ is divisible by $q$. Therefore,
if no square distance is divisible by $q$ then after applying the
above polynomial to every entry of the matrix of square distances 
we get a matrix of full rank modulo $p$, implying the desired
result. This, together with Ramsey's Theorem, also implies that 
for every integer $k$ and every $d$ there is some
$T_0=T_0(k,d)$ so that any set of at least $T_0$
points in  $\mathbb{R}^d $ in which all
square distances are integral determines a square
distance divisible by $k$. To prove it write $k$ as a product 
$k=q_1 q_2 \cdots q_r$ of powers of distinct primes and apply
induction on $r$. For $r=1$ this is the result above for 
prime powers. For $r>1$, by
Ramsey's Theorem and the result for one prime power, any sufficiently
large set of points with all square distances integral 
contains a large subset in which all square distances are divisible
by $q_1$. We can now apply induction to this subset to get in it a
pair of points with square distance divisible by $q_2q_3 \cdots
q_r$, completing the proof. The estimate for $T_0$ here, unlike in
the prime power case, is likely to be very far from being tight.
\item
The assertion of Theorem \ref{points1}  
can be extended
to more general polynomials, including, for
example, the
$\ell_{2r}$ distance raised to the power $2r$ for every even integer
$2r$. More generally, for any fixed polynomial
$P(z)=P(z_1,..,z_d)$ which vanishes at zero, and any set of points
$S=\{x_1, x_2, \ldots ,x_N\}$ in $\mathbb{R}^d$, if $N$ is sufficiently
large as a function of $r$, $d$ and the degree of the polynomial
$P$, then not all the values $P(x_i-x_j)$ for distinct $i,j$ can lie
in a multiplicative group of rank at most $r$. A similar extension
of Theorem \ref{points2} exists as well. The proofs follow the ones of
the above theorems, by considering the $N$ by $N$ matrix
$M_P (S)=(m_{ij})$ defined by $m_{ij}=P(x_i-x_j)$. 
\item
Our final remark, which may well be mentioned somewhere, 
is that by applying 
Theorem \ref{matrix} to the matrix of squares of distances between
pairs of points it follows that for any sequence
$p_1,p_2, \ldots ,p_T$ of $T>{{d+2+s} \choose {d+2}}$  distinct
points in $\mathbb{R}^d$, there is a point $p_i$ determining more than 
$s$ distinct distances from the previous ones.
\end{itemize}

\section{Acknowledgements}
NA is supported in part by NSF grant DMS-1855464 and by BSF grant
2018267.
JS is supported in part by Hungarian National Research Grant KKP 
133819, by NSERC Discovery grant and by OTKA K 119528 grant.

\end{document}